\documentclass[12pt]{amsart}

\usepackage{cancel}
\usepackage{latexsym, amssymb, amsmath,esint}
\usepackage{soul}
\usepackage{amsfonts, graphicx}
\usepackage{graphicx,color}
\newcommand{\be}{\begin{equation}}
\newcommand{\ee}{\end{equation}}
\newcommand{\beq}{\begin{eqnarray}}
\newcommand{\eeq}{\end{eqnarray}}

\usepackage{wasysym,stmaryrd}
\newtheorem{thm}{Theorem}[section]

\newtheorem{que}{Question}
\newtheorem{lma}[thm]{Lemma}
\newtheorem{prop}[thm]{Proposition}

\theoremstyle{remark}
\newtheorem{rem}[thm]{Remark}
\numberwithin{equation}{section}

\def\be{\begin{equation}}
\def\ee{\end{equation}}
\def\bee{\begin{equation*}}
\def\eee{\end{equation*}}

\def\Ric{\text{\rm Ric}}
\def\Rm{\text{\rm Rm}}

\def\e{\varepsilon}

\def\a{{\alpha}}
\def\b{{\beta}}

\def\R{\mathbb{R}}

\begin{document}

\title[]
{On the weakly conical expanding gradient Ricci solitons}

\author{Pak-Yeung Chan}
\address[Pak-Yeung Chan]{Department of Mathematics, National Tsing Hua University, Hsin-Chu, Taiwan}
\email{pychan@math.nthu.edu.tw}

 \author{Man-Chun Lee}
\address[Man-Chun Lee]{Department of Mathematics, The Chinese University of Hong Kong, Shatin, Hong Kong, China}
\email{mclee@math.cuhk.edu.hk}

\renewcommand{\subjclassname}{
  \textup{2020} Mathematics Subject Classification}

\date{\today}

\begin{abstract}
In this work, we construct several sequences of metrics on sphere with different limiting behaviors. By combining with the work of Deruelle, we use it and the localized maximum principle to construct various examples of expanding gradient Ricci solitons with positive curvature and exotic curvature decay. This answers a question proposed by Chow-Lu-Ni and also a question by Cao-Liu, respectively.
\end{abstract}

\maketitle

\markboth{Pak-Yeung Chan, Man-Chun Lee}{On the weakly conical expanding gradient Ricci solitons}

\section{Introduction}\label{s-introduction}

Let $(M,g)$ be a complete Riemannian manifold. It is said to be a gradient Ricci soliton if there is $f\in C^\infty(M)$ and $\lambda\in \mathbb{R}$ such that 
$$\Ric+\frac\lambda2 g= \nabla^2 f.$$
It is said to be expanding, shrinking and steady if $\lambda=1,-1,0$ respectively. This is a natural generalization of the Einstein metric. More importantly, a gradient Ricci soliton generates a self-similar solution to the Ricci flow and has played a predominant role in Perelman's resolution to the Poincar\'e conjecture.  In higher dimensions, Ricci flow on compact manifolds might develop conical singularity. An enormous class of Ricci flow singularities in higher dimensions models on asymptotically conical shrinking Ricci solitons (see \cite{Stolarski2022}).  It has been proposed by \cite{GS2018, BamlerChen2023, AD2024} that asymptotically conical expanding Ricci solitons may be used to continue the flow past the singular time which resolve conical singularities. This motivates us to study the expanding gradient Ricci soliton in depth.

In this work, we would like to study some exotic curvature behaviour of expanding gradient Ricci soliton with positive curvature operator. More precisely, we aim to study the asymptotic scalar curvature ratio:
\[
\text{ASCR}(g):=\limsup_{x\to\infty}\, r^2\cdot  \mathrm{scal}_g
\]
where $r(\cdot)$ denotes the distance from a fixed point. It is not difficult to see that $\text{ASCR}(g)$ is scaling invariant and is well defined independent of the base point of the distance function $r$. $\text{ASCR}(g)$ is finite on any smoothly asymptotically conical manifold and is infinite on any asymptotically cylindrical manifold of positive scalar curvature at infinity. In the presences of curvature non-negativity, the size of $\text{ASCR}(g)$ is also related to gap phenomenon of being flat, for example see \cite{GreeneWu1982}. This is also related to singularity analysis of Ricci flow, we refer readers to \cite{ChowLu2004} for the detailed exposition.

In \cite{CaoLiu2022}, Cao-Liu  showed that a complete $4$-dimensional gradient Ricci expander with nonnegative Ricci curvature and ASCR$(g)<+\infty$ must have finite Riemannian curvature ratio $\sup_Mr^2|\Rm|<+\infty$, and thus is asymptotic to a $C^{1,\a}$ cone at infinity by \cite{ChenDeruelle2015}. See also \cite{CaoLiuXie2023} for a higher dimensional generalization. Chow-Lu-Ni \cite{ChowBook0} asked if the curvature of a positively curved Ricci expander necessarily decays quadratically at infinity. Precisely, the following is asked:
\begin{que}[Ch.9 Sec 7, Q.10 in \cite{ChowBook0}]\label{iascr}
    Does an expanding gradient Ricci soliton with positive curvature operator necessarily have finite asymptotic scalar curvature ratio \text{ASCR}$(g)<+\infty$?
\end{que}
The answer to Question \ref{iascr} is affirmative when $n=2$ by a classification of result of $2$ dimensional gradient Ricci solitons \cite{BM2015}.
In higher dimensions, we demonstrate the existence of two types of positively curved weakly conical Ricci expanders with infinite asymptotic scalar curvature ratio, but with different decay rates. Indeed, we are able to construct expanders with positively infinite asymptotic curvature operator.

\begin{thm} \label{thm:ex1}
Let $n\ge 3$. There exist infinitely many asymptotically conical (in the distance sense) complete expanding gradient Ricci solitons $(M^n,g, f)$ satisfying the following 
\begin{enumerate}
    \item[(a)] positive curvature operator $\Rm>0$;
    \item[(b)] curvature decays at infinity, i.e. $|\Rm|\to 0$ as $x\to \infty$;
    \item[(c)] infinite curvature lower bound ratio, more precisely
 \begin{equation}\label{lambn}
    \limsup_{x\to\infty}r^2 \lambda_n(\Rm)=+\infty,
    \end{equation}
    where $\lambda_n(\Rm)$ denotes the n-th eigenvalue of the curvature operator.
\end{enumerate}
In particular, these examples have infinite scalar curvature ratio.
\end{thm}
\begin{rem} The estimate \eqref{lambn} on $\lambda_n(\Rm)$ is due to the fact that on a cone $(C(X),g_c)$ with smooth link $(X^{n-1}, g_X)$ such that $\Rm(g_X)>1$, where $g_c:=dr^2+r^2g_X$. The kernel of the curvature operator $\Rm(g_c)$ is spanned by elements of the form $\partial_r\wedge v$, where $v\in TX$. Thus $\lambda_j(\Rm(g_c))=0$ for $j=1,\dots,n-1.$
\end{rem}

This answers Question \ref{iascr} in the negative for $n\ge 3$. The first type of expanders has curvature decay to $0$ at infinity and is weakly asymptotic to some singular cones. The links of these cones are chosen carefully so that the corresponding cones are Reifenburg away from the tip and are smooth away from the tip and two edges. We refer the reader to Section \ref{pre} for a brief discussion on asymptotically conical expanders in various senses, as well as the definition of being Reifenburg.

On the other hand, we also construct the second type of expanders which doesn't have any curvature decay at all. The asymptotic cones of these expanders are cones over singular metrics on spheres with two isolated singularities. Thus these asymptotic cones are smooth away from the tip and two edges.

\begin{thm} \label{thm:nodecay} Let $n\ge 3$. There exist infinitely many asymptotically conical (in the distance sense) complete expanding gradient Ricci solitons $(M^n, g, f)$ satisfying the following 
\begin{enumerate}
    \item[(a)] positive curvature operator $\Rm>0$;
    \item[(b)] curvature does not decay at infinity, $$0<\limsup_{x\to\infty}|\Rm|<\infty.$$
\end{enumerate}
In particular, $\text{ASCR}(g)=+\infty$.
\end{thm}

The second line of examples is motivated by a work of Cao-Liu \cite{CaoLiu2022} who showed that for a complete noncompact $4$ dimensional expanding gradient Ricci soliton with nonnegative Ricci curvature, the curvature tensor $\Rm$ must be sufficiently pinched  by the scalar curvature, i.e. $|\Rm|\le C\,\mathrm{scal}_{g}$, if in addition the scalar curvature has at most polynomial decay. Motivating from this, they asked if the scalar curvature of a non-flat gradient expander must have at most polynomial decay, see also  \cite[Ch.9 Sec 7, Q.11]{ChowBook0}:
 \begin{que}[Question 1 in \cite{CaoLiu2022}]\label{CLq}
   Let $(M^n,g,f)$ be an {\sl irreducible}\footnote{In contrast to the original question in \cite{CaoLiu2022}, one further needs to assume the irreducibility of $M$ so as to rule out of the product of lower dimensional expanders.} complete non-compact non-flat expanding gradient Ricci soliton with nonnegative Ricci curvature, where $n\ge 3$. Does the scalar curvature $\mathrm{scal}_g$ have at most polynomial decay, namely \[\liminf_{x\to\infty}r^d \mathrm{scal}_g>0\] for some integer $d\ge 2$?   
 \end{que}
 In real dimension $2$, there exists a family of positively curved gradient expanders with exponential curvature decay and smoothly asymptotic to flat cones with small cone angles \cite{BM2015}. In dimension $n\ge 3$, we construct smoothly asymptotically conical Ricci expander with positive curvature operator and weak curvature decay rate faster than polynomials of any order. In particular, this provides a negative answer to Question \ref{CLq} and generalizes a previous construction in \cite{ChanZhu2022}.
\begin{thm}\label{thm:fast} Let $n\ge 3$. There exist infinitely many smoothly asymptotically conical complete expanding gradient Ricci solitons $(M^n,g, f)$ satisfying the following 
\begin{enumerate}
    \item[(a)] positive curvature operator $\Rm>0$;
    \item[(b)]  do not have at most polynomial curvature decay rate, moreover,
    $$ \liminf_{x\to\infty} d_g(x,p)^m\cdot |\Rm(x)|=0,$$
     for all $m\in \mathbb{N}$.
\end{enumerate}
\end{thm}


The paper is organized as follows. In Section \ref{pre}, we review some basic notions of asymptotically conical expanding gradient Ricci solitons. We collect some recent results on local maximum principles in Section \ref{maxx}.  In Section \ref{sec:sphere-metric}, we construct explicit smooth approximations of some singular metrics on sphere $\mathbb{S}^n$, including both Reifenburg and non-Reifenburg ones. We prove Theorems \ref{thm:ex1}, \ref{thm:nodecay} and \ref{thm:fast} in Section \ref{conssoli}.

\vskip0.2cm

{\it Acknowledgment:} Part of this work was carried out during visits by the first-named author to the Chinese University of Hong Kong in 2024, which we thank for hospitality.  The authors thank Yi Lai for helpful discussions on smoothing singular metrics on spheres.
P.-Y. Chan is partially supported by the Yushan Young Fellow Program of the Ministry of Education (MOE), Taiwan (MOE-108-YSFMS-0004-012-P1), and by a NSTC grant 113-2115-M-007 -014 -MY2. M.-C. Lee was partially supported by Hong Kong RGC grant (Early Career Scheme) of Hong Kong No. 24304222 and No. 14300623, a direct grant of CUHK and a NSFC grant No. 12222122.

\section{Preliminaries}\label{pre}
A complete Riemannian manifold $(M, g)$ with a smooth function $f\in C^{\infty}(M)$ is said to be an expanding gradient Ricci soliton if 
\[
\Ric+\frac12g=\nabla^2 f.
\]
In this case, $f$ is called the potential of the Ricci soliton. 
An expanding gradient Ricci soliton induces an immortal self-similar solution to the Ricci flow. 
Since the vector field $\nabla^g f$ is complete \cite{Zhang2009}, if $\phi_t$ denotes the flow of the 
vector field $-t^{-1}\nabla^g f$ with $\phi_1 = id_M$, then the family of self-similar metrics $g(t):=t \phi_t^*g$ satisfy the Ricci flow equation with $g(1)=g$. We call such $g(t)$ to be the induced (or canonical) Ricci flow from $(M,g,f)$.

We will be interested in two cases, either $\Ric_g\geq 0$ or $\sup_M|\Rm_g|<<1$, which are natural from the viewpoint of Ricci flow existence theory. In both cases, we see that $f$ will be a strictly convex so that $M$ is diffeomorphic to $\R^n$. In this case, $f$ admits a unique critical point $p\in M$. Moreover by the Morse's lemma, the level sets of the potential $\Sigma_s:=\{x\in M: f(x)=s\}$ are smooth hypersurfaces diffeomorphic to $\mathbb{S}^{n-1}$ for all $s>\min_M f$. It is also not difficult to see that for all $s_0>\min_M f$ and all $y\in \bigcup_{s>s_0}\Sigma_s$, there exist unique $t\in (0,1)$ and $x\in \Sigma_{s_0}$ such that $\phi_t(x)=y$. Furthermore by the definition of $g(t)$,
\[
t\cdot |\Rm(g(t))|(x)=|\Rm(g)|(\phi_t(x)).
\]
In case $g$ is of bounded curvature with nonnegative Ricci curvature, it follows from the  Hamilton-Perelman distance distortion \cite[Lemma 3.4]{SimonTopping2022} that for all $0<t\leq 1$
\begin{equation}\label{ddist}
 d_{g}(x,p)\le d_{g(t)}(x,p)=\sqrt{t}d_g(\phi_t(x), p)\leq d_{g}(x,p)+C_0(1-\sqrt{t}).
\end{equation}
for some $C_0>0$. The second inequality also holds if $g$ is only of bounded curvature. In this way, we see that the curvature decay rate of a Ricci expander is reflected by the behavior of $|\Rm(g)|(\phi_t(x))=t|{\Rm}(g(t))|(x)$ of its induced Ricci flow $g(t)$ for $x\in \Sigma_{s_0}$ and all sufficiently small  $t>0$. This observation will play a crucial role.

A complete gradient Ricci expander $(M, g, f)$ is said to be asymptotic to $(C(X),d_c, o)$ {\sl in the distance sense} if for some $x_0\in M$, its canonical induced flow $g(t)$ satisfies 
\begin{equation}\label{cmc}
(M,d_{g(t)},x_o)\xrightarrow{\text{pGH}} (C(X),d_c, o) \text{  as  } t\to 0^+
\end{equation}
in the pointed Gromov Hausdorff sense, and $d_{g(t)}$ converges to some distance function $d_0$ on $M$ as $t\to 0^+$ such that $(M,d_0)$ is isometric to $(C(X),d_c)$.
It can be seen from the distance distortion estimate \eqref{ddist} (see also \cite[Lemma 3.4]{SimonTopping2022}) that a Ricci expander with bounded curvature and nonnegative Ricci curvature with its induced flow $g(t)$ satisfying \eqref{cmc}, is also asymptotic to $(C(X),d_c, o)$ in the distance sense.\\
Let $(C(X),d_c, o)$ be a cone with smooth link, namely $(X,g_X)$ is a smooth closed manifold. Here the distance $d_c$ is induced by the smooth Riemannnian metric $g_c:=dr^2+r^2g_X$.
A complete non-compact expanding gradient Ricci soliton $(M, g, f)$ is said to be smoothly asymptotic to the cone $(C(X), g_c)$, if there exist constants $\varepsilon>0, R_0>0$, and $c_1$, a compact set $K\subseteq M$, and a diffeomorphism $\phi:$   $C(X)\setminus \overline{B(o, R_0)}\longrightarrow M\setminus K$, such that for any nonnegative integer $k$ and all $r> R_0$
\begin{eqnarray}
\label{cona}\sup_{\omega\,\in X}\left|\nabla_{g_c}^k \left(\phi^*g-g_c\right)\right|_{g_c}(r,\omega)&\le& o\left(r^{-k}\right) \text{  as } r\to +\infty;\\
\notag f\circ \phi(r,\omega)&=&\frac{r^2}{4}+c_1.
\end{eqnarray}
In particular, by the conical asymptotic \eqref{cona} of $g$ to $g_c$, such an expander comes out of $(C(X),d_c, o)$ in sense of \eqref{cmc}.
It was shown by Deruelle \cite{Deruelle2017} that a gradient Ricci expander of quadratic Ricci curvature decay with derivatives, i.e. for integer $k\ge 0$, 
\[
    \sup_{M}d_g(x, p_0)^{k+2}|\nabla^k \Ric(g)|(x)<+\infty,
\]
must be smoothly asymptotic to some metric cone with smooth link (see also \cite[Theorem 4.3.1]{Siepmann2013}). 
Let $(X,d)$ be a metric space and $A\subseteq X$. Recall that $A$ is said to be Reifenburg if for any $q\in A$ and any sequence $\lambda_i\to +\infty$, after passing to subsequence if necessary, $(X, \lambda_id, q)$ converges in pointed Gromov Hausdorff sense to the Euclidean space $(\R^n, \delta, o)$, where  $\delta$ is the distance function with respect to standard metric on $\R^n$. See also \cite{DSS2022, DeruelleSimonSchulze2024}.
Throughout this paper, we use $\mathrm{scal}_{g}$ to denote the scalar curvature of a metric $g$.

\section{Maximum principle along Ricci flows}\label{maxx}

In this section, we consider the Ricci flow $g(t)$ coming out of a non-smooth metric space. When the initial data are smooth on an open set $\Omega$, we require the flow $g(t)$ to achieve the initial data smoothly as $t\to 0$ on $\Omega$. This setup appears naturally in the study of geometric compactness problems. It is desirable to control the relevant geometric quantities along the Ricci flow locally in the the regular region. Here we collect some maximum principles established by the second named author and Tam \cite{LeeTam2022} which deal with the incomplete case and singular case.

The following localized maximum principle is proved in
\cite{LeeTam2022}. This roughly says that the subsolution to heat equation could be controlled quantitatively for a short time, along a Ricci flow with scaling-invariant curvature decay. 
\begin{prop}[Theorem 1.1 in \cite{LeeTam2022}]\label{prop:localMP}
Suppose $(M,g(t)),t\in [0,T]$ is a smooth solution to the Ricci flow which is possibly incomplete, such that $\Ric(g(t))\leq \a t^{-1}$ on $M\times (0,T]$ for some $\a>0$. Let $\varphi(x,t)$ be a continuous function on $M\times [0,T]$ such that $\varphi(x,t)\leq \a t^{-1}$ and  
$$(\partial_t-\Delta_{g(t)})\Big|_{(x_0,t_0)}\varphi\leq L\varphi$$
whenever $\varphi(x_0,t_0)>0$ in the sense of barrier, for some continuous function $L$ on $M\times [0,T]$ with $L\leq \a t^{-1}$. If $x_0\in M$ is a point so that $B_{g(0)}(x_0,2)\Subset M$ and $\varphi(x,0)\leq 0$ on $B_{g(0)}(x_0,2)$. Then for any $m\in\mathbb{N}$, there exists $\hat T(n,m,\a)>0$ such that for all $t\in [0,T\wedge \hat T]$,
$$\varphi(x_0,t)\leq t^m.$$
\end{prop}

\begin{rem}
In fact following the method in \cite{LeeTam}, the asymptotic of $\varphi$ in Proposition~\ref{prop:localMP} can be improved to $\exp(-t^{-\delta})$ as $t\to 0^+$, for any $\delta<1$. 
\end{rem}

\section{singular metrics on $\mathbb{S}^m$ and its smoothing}\label{sec:sphere-metric}
In this section, we will construct expanders with different curvature behaviour at infinity. Expanding gradient Ricci solitons coming out of  cones over smooth link $X$ with sufficiently positive curvature are considered by Deruelle using perturbation method. We start by  recalling the  result of Deruelle:

\begin{thm}[\cite{Deruelle2016}]\label{thm:der}
    Let $h$ be a smooth metric on $\mathbb{S}^{n-1}$  satisfying $\Rm(h)\ge 1$ with strict inequality somewhere in $\mathbb{S}^{n-1}$. Then there exists a complete noncompact expanding gradient Ricci soliton with $\Rm>0$ which is smoothly asymptotic to $C(\mathbb{S}^{n-1})$.
\end{thm}

We will use Theorem~\ref{thm:der} with different metric $h$ on $\mathbb{S}^m$ to construct expanders with different behaviors at spatial infinity. 
Let $m\ge 2$. We want to construct smooth positively curved metric metric $h$ on $\mathbb{S}^{m}$ which approximates a Riefenburg metric on $\mathbb{S}^{m}$ of low regularity and unbounded curvature. To do this, we will look for a rotationally symmetric metric $h$ on $(0,L)\times \mathbb{S}^{m-1},L>0$ of the form
\[
h:=dr^2+a(r)^2\, g_{\mathbb{S}^{n-1}},
\]
where $a:(0,L)\to (0,\infty)$ is a smooth positive function and $g_{\mathbb{S}^{m-1}}$ is the standard metric with constant curvature $1$ on the unit sphere $\mathbb{S}^{m-1}$ for $m\ge 3$ or the flat metric on $\mathbb{S}^1$ for $m=2$. The metric can be made complete by adding two points at $r=0$ and $r=L$ respectively, and becomes a smooth metric on $\mathbb{S}^{m}$ if $a$ is a smooth function on $[0,L]$ and satisfies the following boundary conditions (see \cite[Section 1.4.4, Section 4.3.4]{PetersenBook}):
\begin{equation}\label{bdry}
  a^{(2\ell)}(0)=0,\qquad a'(0)=1,\qquad a^{(2\ell)}(L)=0,\qquad a'(L)=-1,  
\end{equation}
for all nonnegative integer $\ell$. The eigenvalues of $\Rm(h)$ are given by $-a''/a$ and $(1-(a')^2)/a^2$ for $m\ge 3$, and solely by $-a''/a$ for $m=2$, see \cite[Section 4.2.3]{PetersenBook}. Therefore, in order to ensure $\Rm(h)\geq 1$, we require 
\begin{equation}\label{eqn:curva} 
\left\{
\begin{array}{ll}
\displaystyle -\frac{a''(r)}{a(r)}&\ge 1 \\[4mm]
 \displaystyle\frac{1-(a'(r))^2}{a^2(r)}&\ge 1, \,\,\;\text{if}\;\; n\geq 3
\end{array}
\right.
\end{equation}
for any $r\in (0,L)$. 

\medskip
\subsection{Smoothing a Reifenburg metric}\label{subsec:Reifenburg}

We want to construct a sequence of metrics on sphere which has $\Rm\geq 1$, spherical somewhere and increasingly singular around the spherical point. On $\mathbb{S}^m$, we will parametrize it using warped product coordinate as described above. For $k>1$, we let $L=\pi/k$ and we look for a metric $h=dr^2+a^2(r)\,g_{\mathbb{S}^{m-1}}$ on $(0,\pi/k)\times \mathbb{S}^{m-1}$. The basic tool we use is the warping function $$\b_{\e,\ell}(r)=(1-\e r)\sin(\ell r)/\ell$$ where $\e\geq 0$ and $\ell\geq 1$. We will eventually choose the warped product function $a(r)$ which transited from $\b_{0,1}$ to  $\b_{\e,k}$ (approximately) and finally to $\b_{0,k}$ for $\e\geq 0$ and $k>1$. The boundary conditions $\b_{0,1}$ and $\b_{0,k}$ are imposed so that the resulting metric $h$ is smooth on $\mathbb{S}^m$ while $\b_{\e,k}$ with $\e>0$ corresponds to singular metrics at $r=0$ on $\mathbb{S}^m$ but is still Reifenburg there.

We now make it precise. Fix a smooth non-increasing function $\phi$ on $[0,+\infty)$ which is identical to $1$ on $[0,1]$ and vanishes outside $[0,2]$. For $\lambda>0$, we will denote $\phi_\lambda(r)=\phi(r/\lambda)$ for convenience. For $\e\geq 0,k>1$  and $ \pi/8k>\delta>0$, consider the ODE:
\begin{equation*}
\left\{
\begin{array}{ll}
\gamma_{\delta,\e}'=\phi_\delta \b_{0,1}'+(1-\phi_\delta)\b_{\e,k}';\\[3mm]
\gamma_{\delta,\e}(0)=0
\end{array}
\right.
\end{equation*}
on $[0,\pi/k]$. It follows from our choice of $\phi_\delta$ that $\gamma_{\delta,\e}(r)\equiv \b_{0,1}(r)$ on $[0,\delta]$. 
Finally, we truncate it back to $\b_{0,k}$ nearby $r=\pi/k$ by defining 
\begin{equation*}
a_{\delta,\e}=\phi_{\pi/8k}\cdot \gamma_{\delta,\e}+(1-\phi_{\pi/8k}) \cdot \b_{0,k}.
\end{equation*}
We denote the resulting metric on $\mathbb{S}^m$ by $h_{\delta,\e}=dr^2+a_{\delta,\e}^2(r)\,g_{\mathbb{S}^{m-1}}$. This is a smooth metric on $\mathbb{S}^m$ for $\delta\in (0,\pi/4k)$, since 
\begin{enumerate}
\item[(i)] $a_{\delta,\e}(r)=\sin(r)$ for $r\in [0,\delta]$;
\item[(ii)]$a_{\delta,\e}(r)=\sin(kr)/k$ for $ r\in [\pi/4k,\pi/k]$.
\end{enumerate}

We start with some preliminary computations. 
\begin{lma} \label{lma:ode-compare}
There exists $\e_0(k)>0$ such that for all $r\in(0,\pi/4k)$ and $\e\in [0,\e_0)$,
\begin{equation*}
\begin{split}
\b_{\e,k}'&=\cos(kr) -\e \left[r\cos(kr)+\frac1k \sin(kr) \right]>0;\\
\b_{\e,k}''&=-k\sin(kr)-\e \left[2\cos(kr)-kr \sin(kr)\right]<0.
\end{split}
\end{equation*}
In particular, $0<\b_{\e,k}'\leq \b_{0,k}'$ and $\b_{\e,k}''\leq \b_{0,k}''<0$ on $(0,\pi/4k)$.
\end{lma}
\medskip

We now claim that $h_{\e,\delta}$ has the desired curvature lower bound, for suitable choices of $\delta,\e$. We will split our discussion into several regions. 

\begin{prop}\label{prop:sphere-1}
For $\e\in [0,\e_0)$, $h_{\delta,\e}=dr^2+a_{\delta,\e}^2(r)g_{\mathbb{S}^{m-1}}$ is a smooth metric on $\mathbb{S}^m$ satisfying 
\begin{enumerate}
\item[(a)]
$\mathrm{sec}(h_{\delta,\e})\equiv 1$ on   $(0,\delta]\times \mathbb{S}^{m-1}$;
\item[(b)] $\Rm(h_{\delta,\e})\geq 1$ on $(\delta,2\delta]\times \mathbb{S}^{m-1}$.
\end{enumerate}
\end{prop}
\begin{proof}
By the construction, $h_{\delta,\e}$ is a smooth metric on $\mathbb{S}^m$ such that  $h_{\delta,\e}$ is spherical on $(0,\delta]$. Therefore, it suffices to verify \eqref{eqn:curva} on $[0,2\delta]$.
On $[0,2\delta]$,  we have $a_{\delta,\e}=\gamma_{\delta,\e}$ in which $
0\leq \gamma'_{\delta,\e}\leq \b_{0,1}'$
thanks to Lemma~\ref{lma:ode-compare} and $\b_{0,k}'\leq \b_{0,1}'$ for $0\leq r\leq \pi/4k$. Since $\gamma_{\delta,\e}(0)=\b_{0,1}(0)$, we have 
\begin{equation}\label{eqn:gamma-1}
0\leq \gamma_{\delta,\e}\leq \b_{0,1}
\end{equation}
 on $[0,2\delta]$. In particular, 
\begin{equation}
\begin{split}
1-(\gamma_{\delta,\e}')^2&\geq 1-(\b_{0,1}')^2=\b_{0,1}^2\geq (\gamma_{\delta,\e})^2.
\end{split}
\end{equation}
This justifies the first inequality. 
For the second one, taking derivatives to the ODE yields
\begin{equation}
\begin{split}
-\gamma_{\delta,\e}''&=-\b_{\e,k}''-\phi_\delta' (\b_{0,1}'-\b_{\e,k}')-\phi_\delta (\b_{0,1}''-\b_{\e,k}'')\\
&\geq -\b_{\e,k}''-\phi_\delta (\b_{0,1}''-\b_{\e,k}'')\\
&\geq -(1-\phi_\delta) \b_{0,k}''-\phi_\delta \b_{0,1}''\\
&=k^2(1-\phi_{\delta})\b_{0,k}+\phi_\delta \b_{0,1}\\
&\geq \b_{0,1}\geq \gamma_{\delta,\e}
\end{split}
\end{equation}
where we have used $\phi_\delta'\leq 0$ and Lemma~\ref{lma:ode-compare}.
\end{proof}
\medskip

We now wish to obtain a more precise curvature lower bound for $r\geq 2\delta$. To do this, we need a better upper bound on $\gamma_{\delta,\e}$ and $a_{\delta,\e}$, provided that $\delta$ is small enough.
\begin{lma}\label{lma:bdd-gamma}
For any $\sigma>0$, there exists $\delta_0(k,\sigma)>0$ such that if $\delta\in (0,\delta_0)$, then 
\begin{enumerate}
\item[(a)] $\gamma_{\delta,\e}\leq \b_{0,1}$ for $r\in [0,2\delta]$ and;
\item[(b)]$\gamma_{\delta,\e}\geq \b_{\e,k}$ for $r\in [0,2\delta]$ and;
\item[(c)]$\b_{\e,k}\leq \gamma_{\delta,\e}\leq (1+\sigma)\b_{0,k}$ for $r\in [2\delta,\pi/4k]$.
\end{enumerate}
\end{lma}
\begin{proof}
The conclusion (a) follows from \eqref{eqn:gamma-1} while (b) follows from $\gamma_{\delta,\e}(0)=\b_{\e,k}(0)=0$ and 
$$\gamma_{\delta,\e}'\geq \b_{\e,k}',\quad \;\;\forall r\in [0,\pi/4k).$$

It remains to show (c). Recall from Lemma~\ref{lma:ode-compare} that for $\pi/4k\geq r\geq 2\delta$ and $\e\geq 0$, $\gamma_{\delta,\e}$ satisfies 
\begin{equation}
\begin{split}
\gamma'_{\delta,\e}=\b_{\e,k}'&=\cos(kr) -\e \left[r\cos(kr)+\frac1k \sin(kr) \right]\\
&\leq \b_{0,k}'< (1+\sigma)\b_{0,k}'
\end{split}
\end{equation}
where we have used $2\cos(x)\geq x \sin(x)$ for $x\in (0,\pi/4]$. Using (a), we see that there exists $\delta_0(k,\sigma)>0$ such that if $\delta\in(0,\delta_0)$,
\begin{equation}
\begin{split}
\gamma_{\delta,\e}(2\delta)&\leq \b_{0,1}(2\delta)=\sin(2\delta)\\
&\leq (1+\sigma)\cdot \sin(2k\delta)/k=(1+\sigma)\b_{0,k}(2\delta).
\end{split}
\end{equation}
The upper bound in (c) follows from integration from $r=2\delta$. The lower bound follows from the fact that $\gamma_{\delta,\e}(2\delta)\geq \b_{\e,k}(2\delta)$ and $\gamma_{\delta,\e}'=\b_{\e,k}'$ on $[2\delta,\pi/4k]$. This completes the proof.
\end{proof}

We now estimate the curvature lower bound.
\begin{prop}\label{prop:sphere-2}
For all $\sigma>0$, there exists $\delta_0(k,\sigma),\e_0(k,\sigma)>0$ such that if $\delta\in(0,\delta_0)$, then the metric $h_{\delta,\e}$ is smooth on $\mathbb{S}^m$ satisfying 
\begin{enumerate}
\item[(a)]$\Rm(h_{\delta,\e})\geq k^2(1+\sigma)^{-1}$ on $[2\delta,\pi/4k]\times \mathbb{S}^{m-1}$;
\item[(b)] $\mathrm{sec}(h_{\delta,\e})\equiv k^2$ on $[\pi/4k,\pi/k)\times \mathbb{S}^{m-1}$.
\end{enumerate}
Together with Proposition~\ref{prop:sphere-1}, we have $\Rm(h_{\delta,\e})\geq 1$ on $\mathbb{S}^m$.
\end{prop}
\begin{proof}
On $[2\delta,\pi/4k]$, $\gamma'_{\delta,\e}=\b_{\e,k}'$ so that 
\begin{equation}\label{eqn:1d-a}
\begin{split}
a_{\delta,\e}'&=\phi_{\pi/8k}'\cdot (\gamma_{\delta,\e}-\b_{0,k})+\phi_{\pi/8k}(\gamma_{\delta,\e}'-\b_{0,k}')+\b_{0,k}'\\
&= \phi_{\pi/8k}'\cdot (\gamma_{\delta,\e}-\b_{\e,k})-\e \cdot\phi_{\pi/8k}\cdot \left[r\cos(kr)+\frac{\sin(kr)}{k} \right]\\
&\quad -\phi_{\pi/8k}' \frac{\e r\cdot \sin(kr)}{k}+\cos(kr).
\end{split}
\end{equation}

Fix a sufficiently small $\sigma>0$ and obtain $\delta_0(k,\sigma)$ using Lemma~\ref{lma:bdd-gamma}. Consequently, we see that if $\sigma,\e_0$ are small enough, then $a_{\delta,\e}'>0$ on $[2\delta,\pi/4k]$ if $\delta\in (0,\delta_0)$ and $\e\in [0,\e_0)$. Furthermore using Lemma~\ref{lma:bdd-gamma} again, it gives
\begin{equation}
0<a_{\delta,\e}'\leq -\phi_{\pi/8k}' \frac{\e r\cdot \sin(kr)}{k}+\cos(kr)\leq \cos(kr)+C_0k^{-1}\e r^2
\end{equation}
for some universal constant $C_0>0$. Therefore, 
\begin{equation}
\begin{split}
1-(a_{\delta,\e}')^2&\geq \sin^2(kr)-C_1\e r^2\\
&=(k^2-C_2 \e) \b_{0,k}^2
\end{split}
\end{equation}
for some $C_2>0$ where we have used a rough bound: $\sin(x)/x\geq 1/\sqrt{2}$ for $x\in [0,\pi/4]$. On the other hand, Lemma~\ref{lma:bdd-gamma} implies 
\begin{equation}
\begin{split}
a_{\delta,\e}&=\phi_{\pi/8k}(\gamma_{\delta,\e}-\b_{0,k})+\b_{0,k}\leq (1+\sigma)\b_{0,k}.
\end{split}
\end{equation}

Combine two inequalities, we conclude that 
\begin{equation}
\begin{split}
\frac{1-(a_{\delta,\e}')^2}{a_{\delta,\e}^2}\geq \frac{k^2-C_2  \e}{(1+\sigma)^2}. 
\end{split}
\end{equation}

Similarly by differentiating \eqref{eqn:1d-a} once more, 
\begin{equation}
\begin{split}
-a_{\delta,\e}''&=-\phi_{\pi/8k}''\cdot (\gamma_{\delta,\e}-\b_{\e,k})+\e \cdot\phi_{\pi/8k}'\cdot \left[r\cos(kr)+\frac{\sin(kr)}{k} \right]\\
&\quad+\e \cdot\phi_{\pi/8k}\cdot \left(2\cos(kr)-kr \sin(kr)\right)\\
&\quad  +\phi_{\pi/8k}'' \frac{\e r\cdot \sin(kr)}{k}+\phi_{\pi/8k}' \left[\frac{\e r\cdot \sin(kr)}{k}\right]'+k^2\b_{0,k}\\
&\geq (k^2-C_3\e-C_3\sigma)\b_{0,k}.
\end{split}
\end{equation}
for some $C_3>0$, where we have used 
\begin{equation*}
\begin{split}
\b_{\e,k}\leq \gamma_{\delta,\e}&\leq (1+\sigma) \b_{0,k}=\frac{1+\sigma}{1-\e r} \b_{\e,k}
\end{split}
\end{equation*}

This shows 
\begin{equation}
-\frac{a_{\delta,\e}''}{a_{\delta,\e}}\geq \frac{k^2-C_3\e-C_3\sigma}{1+\sigma}.
\end{equation}

The assertion follows by choosing $\e,\sigma$ small enough with re-labelling.
\end{proof}

By letting $\delta\to 0$, we obtain a sequence of metrics on $\mathbb{S}^m$ whose limit is non-smooth at $r=0$ with unbounded curvature, but is Reifenburg at $r=0$.
\begin{prop}\label{prop:sphere-3}
For any $k>1$ and $\e\in (0,\e_0)$,  there exists a sequence of smooth metrics $h_i$ on $\mathbb{S}^m$ such that $\Rm(h_i)\geq 1$ and it converges in Gromov-Hausdorff sense as $i\to+\infty$ to a Reifenburg metric $h_\infty$ on $\mathbb{S}^m$ given by
$$h_\infty=dr^2+a_{0,\e}^2(r)\,g_{\mathbb{S}^{m-1}},$$
where $a_{0,\e}(r)=\phi_{\pi/8k}\cdot \b_{\e,k}+(1-\phi_{\pi/8k}) \cdot \b_{0,k}$, in warped product coordinate, $(0,\pi/k)\times \mathbb{S}^{m-1}$, such that it is smooth outside $r=0$ in the warped product coordinate and satisfies $\lim_{r\to 0}\lambda_1(\Rm(h_i))=+\infty$.
\end{prop}
\begin{proof}
Fix $k>1$ and $\e$ small. Let $\delta_i\to 0$ and $h_{\delta_i,\e}$ be the metric obtained above. Clearly, $\lim_{i\to \infty}\gamma_{\delta_i,\e}(r)=\lim_{i\to 0}\int_0^r \phi_{\delta_i}\b_{0,1}'+(1-\phi_{\delta_i})\b_{\e,k}'\,ds=\b_{\e,k}(r)$. Hence $h_{\delta_i,\e}\to h_{0,\e}=dr^2+a_{0,\e}^2(r)\,g_{\mathbb{S}^{m-1}}$ where $$a_{0,\e}(r)=\lim_{i\to \infty}\gamma_{\delta_i,\e}(r)=\b_{\e,k}(r)$$
for $r\in (0,\pi/8k]$ and similarly $$a_{0,\e}(r)=\lim_{i\to \infty}\phi_{\pi/8k}\cdot \gamma_{\delta_i,\e}+(1-\phi_{\pi/8k}) \cdot \b_{0,k}=\phi_{\pi/8k}\cdot \b_{\e,k}+(1-\phi_{\pi/8k}) \cdot \b_{0,k}$$ for $r\in (\pi/8k, \pi/k)$. Clearly, $h_{0,\e}$ is singular at $r=0$, and smooth away from $r=0$. Its blow-up model is $\mathbb{R}^n$ while the assertion of $\lambda_1(\Rm(h_{0,\e}))$ follows from direct computation using Lemma~\ref{lma:ode-compare}. 
\end{proof}

\subsection{Smoothing a non-Reifenburg metric}
In this subsection, we construct sequence of metrics on $\mathbb{S}^m$ in which its limit is non-Reifenburg somewhere. We follow closely the discussion in sub-section~\ref{subsec:Reifenburg}. We work on the coordinate $(0,\pi)\times \mathbb{S}^{m-1}$ and will look for rotationally symmetric metric $h=dr^2+a(r)^2 \,g_{\mathbb{S}^{m-1}}$. As in sub-section~\ref{subsec:Reifenburg}, we fix a smooth non-increasing function $\phi:[0,+\infty)\to [0,1]$. For $\delta\in (0,\pi/4)$ and $k>1$, consider the ODE:
\begin{equation*}
\left\{
\begin{array}{ll}
\gamma_{\delta}'=\phi_\delta \b_{0,1}'+\frac1k(1-\phi_\delta)\b_{0,1}';\\[3mm]
\gamma_{\delta}(0)=0
\end{array}
\right.
\end{equation*}
on $[0,\pi/2]$, where $\b_{0,1}(r)=\sin(r)$, under the notations in sub-section~\ref{subsec:Reifenburg}. We extend it by defining $\gamma_\delta(r)=\gamma_{\delta}(\pi-r)$ for $r\in [\pi/2,\pi]$, i.e. the reflection along the axis $r=\pi/2$. The extension is smooth since $\gamma_\delta'(r)=k^{-1}\cos(r)$ for $r\in [2\delta,\pi/2]$. Therefore, the resulting metric  $h_{\delta}:=dr^2+ \gamma_\delta^2(r)\,g_{\mathbb{S}^{m-1}}$ is a smooth metric on $\mathbb{S}^{m}$ for $\delta\in (0,\pi/4)$. Furthermore, we have 
\begin{equation}
\gamma_0(r)=\lim_{\delta\to 0} \gamma_\delta(r) =\int^r_0 \phi_\delta \b_{0,1}'+\frac1k(1-\phi_\delta)\b_{0,1}'\, ds=\frac1k \sin(r)
\end{equation}
where $h_0=dr^2+\gamma_0^2 \, g_{\mathbb{S}^{m-1}}$ is a smooth metric outside $r=0,\pi$ and its tangent cone at $r=0$ is non-Euclidean if $k>1$.  This is inspired by some related construction by Lai \cite{Lai2020}.
\begin{prop}\label{prop:sphere-4}
The metric $\hat h_{\delta}:=k^{-1}h_\delta$ satisfies $\Rm(\hat h_\delta)\geq 1$. Moreover as $\delta\to0$, $(\mathbb{S}^m,\hat h_{\delta})$ converges in Gromov-Hausdorff sense to $(\mathbb{S}^m,\hat h_0)$ which is smooth outside $r=0,\pi$ and is not Reifenburg at $r=0,\pi$.
\end{prop}
\begin{proof}
It suffices to consider $r\in (0,\pi/2)$ where we have $
0<\frac1k \b_{0,1}'\leq \gamma_\delta'\leq  \b_{0,1}'$
for all $\delta>0$. Therefore, 
\begin{equation}
\frac1k \b_{0,1}(r)\leq \gamma_\delta(r)\leq  \b_{0,1}(r)
\end{equation}
for all $r\in (0,\pi/2)$ and $\delta>0$.  Thus,
\begin{equation}
\begin{split}
\frac{1-(\gamma_\delta')^2}{\gamma_\delta^2}&\geq \frac{1-(\b_{0,1}')^2}{\b_{0,1}^2}=1.
\end{split}
\end{equation}
Similarly using $\phi'\leq 0$, we have 
\begin{equation}
\begin{split}
-\gamma_\delta''&=\frac1k \sin(r)-\left(1-\frac1k \right)\phi_\delta' \cos(r)+\left(1-\frac1k \right)\phi_\delta \sin(r)\\
&\geq \frac1k \sin(r) +\left(1-\frac1k \right)\phi_\delta \sin(r)\geq \frac1k  \gamma_\delta.
\end{split}
\end{equation}
Hence, we see that $\Rm(h_\delta)\geq \frac1k$.
\end{proof}

\section{Construction of solitons}\label{conssoli}

In this section, we will use Theorem~\ref{thm:der} and metrics constructed in section~\ref{sec:sphere-metric} to produce different gradient expanding Ricci solitons with $\Rm\geq 0$. We need a compactness result which is well-known to experts. See for instance \cite{SchulzeSimon2013, Lai2020, DSS2022}.

\begin{prop}\label{prop:appr}
    Let $(M_i,g_i, f_i)$ be a sequence of complete noncompact expanding gradient Ricci solitons with $\Rm(g_i)\geq 0$ and $\mathrm{AVR}(g_i)\ge c_0>0$ for some constant $c_0$. Suppose that each $(M_i,g_i, f_i)$ is smoothly asymptotic to $(C(X_i), dr^2+r^2g_{X_i}, o_i)$ where $(X_i,g_{X_i})$ is a smooth closed manifold with $\Rm(g_{X_i})\ge 1$ and $(X_i,d_{g_{X_i}})$ converges in Gromov Hausdorff sense to a compact metric space $(X_\infty, d_{g_{X_\infty}})$. Then after passing to a subsequence, $(M_i,g_i(t), p_i)$ converges in Cheeger Gromov sense on $(0,\infty)$ as $i\to \infty$ to a complete expanding gradient Ricci soliton $(M_\infty,g_\infty(t), p_\infty)$ which is asymptotic to $C(X_\infty)$ in the distance sense, where $g_i(t)$ is the induced Ricci flow $g_i$ and $p_i$ is the unique critical point of $f_i$. 
\end{prop}
\begin{proof}
    We sketch the idea of proof for the sake of completeness. By the work of Perelman \cite{Perelman2002}, $\Rm(g_i)\geq 0$ and volume non-collapsing, there exists a positive constant $C_1=C_1(n,c_0)$ such that for all $i$
    \[
    |{\Rm(g_i(t))}|\leq C_1t^{-1}\;\; \text{  on  } M_i\times (0,\infty).
    \]
    By the Hamilton compactness \cite{Hamilton-compact}, the Cheeger-Gromov convergence to a complete immortal Ricci flow $(M_\infty,g_\infty(t), p_\infty)$ follows. We now claim that $g_\infty(t)$ is an induced Ricci flow of an expanding gradient Ricci soliton.  If suffices to consider $t=1$.     By the Hamilton's identity $|\nabla f|^2+R=f$, we have $0\le f_i(p_i)\le R_{g_i}(p_i)\leq C_2(n,c_0)$. By the uniform quadratic growth of $f_i$ from \cite[Corollary 27.10]{ChowBook4}, $|\nabla_{g_i}f_i|^2\le |f_i| \leq C(r)$ on $B_{g_i}(p_i,r)$ for all $i$ and $r>0$. Thanks to soliton's equation and curvature estimate of $g_i(t)$, $f_i$ is locally uniform bounded in $C^k_{loc}$ and hence $f_i$ sub-converges smoothly and locally to $f_\infty$ on $M_\infty$ which satisfies
    \[
    2\nabla_{g_\infty(1)}^{2}f_\infty=2\Ric(g_\infty(1))+g_\infty(1),
    \]
    where $g_\infty(t)$ is the limiting Ricci flow obtained from above. It remains to prove the asymptotic.
    By the definition of cone metric, we have $(C(X_i), d_{c_i},o_i)$ converges in pointed Gromov Hausdorff sense to $(C(X_\infty), d_{c_\infty}, o_\infty)$ as $i\to\infty$. On the other hand by the Hamilton-Perelman distance distortion estimate (e.g. see \cite[Lemma 3.4]{SimonTopping2022}), we have 
    \[
    d_{g_i(t)}\leq d_{C(X_i)}\leq  d_{g_i(t)}+C\sqrt{t} 
    \]
    for some $C(n,c_0)>0$, since $g_i(t)$ is also a Ricci flow coming out of its asymptotic cone. Now it follows that   $d_{g_\infty(t)}$ converges to a metric $d_{\infty,0}$ generating the same topology as $M_\infty$, by  \cite[Lemma 3.1]{SimonToppingGT}, as $t\to 0^+$ such that $(M_\infty,d_{\infty,0})$ is isometric to $(C(X_\infty), d_{c_\infty})$. This proves the proposition.
    \end{proof}

\medskip

We now use the compactness to construct expanding gradient Ricci solitons coming out of cone $C(X)$ with non-smooth link $X$.
\medskip

We start with the construction showing expanding soliton with curvature decay at infinity but has infinite scalar curvature ratio. 
We need a lemma relating the curvature decay of an expander and the Riefenburg property of its asymptotic cone.
\begin{lma}\label{curdecay} Let $(M^n,g,f)$ be a non-flat complete expanding gradient Ricci soliton with $\Ric \ge 0$ and bounded curvature, and asymptotic to a cone $C(X)$ in the distance sense. Then $\displaystyle\lim_{x\to\infty}|{\Rm}|=0$ if and only if $C(X)\setminus \{o\}$ is Reifenburg,
\end{lma}
\begin{proof} 
    Suppose that $\lim_{x\to\infty}|{\Rm(g)}|(x)=0$. We first show that $M\setminus\{p\}$ is Riefenburg in $(M,d_0)\cong (C(X),d_c)$, where $p$ is the unique critical point of $f$. This implies that $p=o$ and result follows. Indeed, by Colding volume convergence theorem \cite{CheegerColding1997} and the rigidity in volume comparison, $C(X)$ cannot be isometric to $\R^n$. If $p\neq o$, then $o\in M\setminus\{p\}$ and the tangent cone at $o$ w.r.t. $d_0$ is $\R^n\cong C(X)$, which is impossible. Here we have used the fact that the tangent cone of the metric cone $C(X)$ at the tip $o$ is $C(X)$ itself in the last identity.
    
    For any $z\in M\setminus\{p\}$, there is a small $r_0>0$ such that $B_{d_0}(z,r_0)\subset \subset  M\setminus\{p\}$. By \eqref{ddist}, for all $\rho>0$, there is $1>t_0>0$ such that $d_g(p,\phi_t(B_{d_0}(z,r_0))>\rho$ for all $t\in (0,t_0)$. Hence by $|\Rm(g)|(\phi_t(x))=t|{\Rm}(g(t))|(x)$ and the condition that $|{\Rm}(g)|$ decays at infinity, we have
   \begin{equation}\label{supcrit}
         \lim_{t\to 0^+}\,\varepsilon(t):= \lim_{t\to 0^+}\sup_{x\in B_{d_0}(z,r_0)} t|{\Rm}(g(t))|(x)=0,
   \end{equation} where $g(t)$ is the canonical immortal Ricci flow generated by the expander. For any sequence of positive numbers $\lambda_i\to\infty$, let $h_i(t):=\lambda_i g(t/\lambda_i)$. By the distance distortion, for all $r>0$, there is an integer $N(r)$ such that for all $i\ge N(r)$, $B_{h_i(1)}(z, r)\subseteq B_{d_0}(z,(r+c_n)/\sqrt{\lambda_i})$. Hence for $(x,t)\in B_{h_i(1)}(z, r)\times (r^{-2}, r^2)$, for all large $i$
    \[
    |{\Rm}(h_i(t))|(x)=\left|{\Rm}\left(g\left(t/\lambda_i\right)\right)\right|/\lambda_i\le \frac{\varepsilon\left(t/\lambda_i\right)}{t}.
    \]
     Since AVR$(h_i(t))=$ AVR$(g(t/\lambda_i))=$ AVR$(g)>0$ \cite{Yok}, we can apply the Cheeger-Gromov-Taylor injectivity radius estimate and the Hamilton compactness  \cite{Hamilton-compact} 
      as well as 
     the supercritical decay above to see that $(M, h_i(t),z)$ converges smoothly in Cheeger-Gromov sense on $(0,\infty)$ to a flat flow with initial data attained in the distance sense given by a tangent cone $C_0$ at $z$ of $(M, d_0)$. Thanks to AVR$(g)>0$, the limiting flow is the static $\R^n$, and $C_0$ is isometric to $\R^n$ by the distance distortion. Hence any tangent cone at $z$ w.r.t $d_0$ is $\R^n$.

    Another direction follows from Deruelle-Simon-Schulze \cite[Lemma A.2]{DSS2022}. Indeed, suppose for the sake of contradiction that there is a sequence $y_i\to\infty$ in $M$ such that $\varepsilon_0:=\limsup_{i\to \infty}|{\Rm(g)}|(y_i)>0$. We then choose a large $s_0>\min_M f$ such that $o\not\in \Sigma_{s_0}$. For all large $i$, by \eqref{ddist}, there are sequence of $z_i\in \Sigma_{s_0}$ and $0< t_i\to o^+$ as $i\to +\infty$ such that $y_i=\phi_{t_i}(z_i)$, where $\phi_{t}$ is the flow of $-\nabla f/t$ with $\phi_1=\text{id}$. After passing to subsequence, we may assume that $z_i\to z_\infty\in \Sigma_{s_0}$ as $i\to +\infty$. There is a small positive $r_1\ll 1$ such that $o\not\in B_{d_0}(z_\infty,r_1)$. By shrinking $r_1$ if necessary, we then apply \cite[Lemma A.2]{DSS2022} to see that there exists $\hat \delta>0$ such that for all $(x,t)\in B_{d_0}(z_\infty,r_1)\times (0,\hat\delta)$
    \[
    |{\Rm(g(t))}|(x)\leq \frac{\varepsilon_0}{4t}.
    \]
    In particular, for all large $i$, $z_i\in B_{d_0}(z_\infty,r_1)$ and $t_i\in (0,\hat\delta)$
    \[
  0<  \varepsilon_0/2\le |{\Rm(g)}|(y_i)=|{\Rm(g)}|(\phi_{t_i}(z_i))=t_i|{\Rm(g(t_i))}|(z_i)\le \varepsilon_0/4.
    \]
    Contradiction. This proves the lemma.
\end{proof}

\begin{proof}[Proof of Theorem~\ref{thm:ex1}]
Fix any $k>1$ and small $\varepsilon>0$. We obtain a sequence of metric $h_i$ on $\mathbb{S}^{n-1}$ using Proposition~\ref{prop:sphere-3}. Clearly, $(\mathbb{S}^{n-1},d_{h_i})$ converges in Gromov-Hausdorff sense to $(\mathbb{S}^{n-1},d_{h_\infty})$ as $i\to +\infty$. For each $i\in \mathbb{N}$,  we apply Theorem \ref{thm:der} to obtain a Ricci expander $(M_i,g_i, f_i)$ with $\Rm(g_i)\geq 0$, smoothly asymptotic to $(C(\mathbb{S}^{n-1}), dr^2+r^2h_i, o)$.  By Proposition~\ref{prop:appr}, $(M_i,g_i, f_i,p_i)$ converges in smooth Cheeger-Gromov sense to an expander  $(M_\infty,g_\infty, f_\infty,p_\infty)$ with $\Rm(g_\infty)\ge 0$ which is asymptotic to $\left( C(\mathbb{S}^{n-1}),g_{c_\infty}=dr^2+r^2h_\infty,o\right)$ in the distance sense. It can be seen that $C(\mathbb{S}^{n-1})\setminus\{o\}$ is Reifenburg w.r.t. $d_{g_{c_\infty}}$.  Hence Lemma \ref{curdecay} implies $|{\Rm(g_\infty)}|(x)\to 0$ as $x\to \infty$. Let $q_0\in\mathbb{S}^{n-1}$ be the point which corresponds to the isolated singularity of $h_{\infty}=ds^2+a_{0,\e}^2(s)\,g_{\mathbb{S}^{n-2}}$, namely $s=0$. It can be seen that $g_{c_\infty}$ is a smooth Riemannian metric on $N:=C(\mathbb{S}^{n-1})\setminus\{o_\infty, (r, q_0), r>0\}$. It follows from \cite[Theorem 1.6]{DSS2022b} that $g(t)$ converges in $C^{\infty}_{loc}(N)$ to $g_{c_\infty}$ thanks  to the distance convergence $d_{g(t)}\to d_{g_{c_\infty}}$ as $t\to 0^+$. Moreover by Proposition~\ref{prop:sphere-2}, $\Rm(h_i)$ has lower bound strictly greater than $1$ on a neighbourhood of $s=\pi/k$. Thus, a strong maximum principle argument, see  \cite[Lemma 2.2]{GS2018} for example, gives  $\Rm(g_\infty(t))>0$ on $M_\infty$ for all $t>0$.

It remains to control the asymptotic curvature lower bound ratio. By Proposition~\ref{prop:sphere-3}, for any $L>0$, there exists $q_L\in  \mathbb{S}^{n-1}\setminus \{q_0\}$ such that $\lambda_n(\Rm_{g_{c_\infty}})\geq L$ at $p'_L=(1,q_L)\in (0,+\infty)\times \mathbb{S}^{n-1}$. Using the smooth convergence of $g_\infty(t)$ to $g_{c_\infty}$ and distance distortion \cite[Lemma 3.4]{SimonTopping2022}, we have  
\begin{equation}
\begin{split}
 &\quad  \lim_{t\to 0} d_{g_{\infty}}(\phi_t(p_L'),p_\infty)^2\cdot \lambda_n(\Rm_{g_\infty})(\phi_t(p_L'))\\
 &=\lim_{t\to 0} d_{g_\infty(t)}(p_L',p_\infty)^2\cdot \lambda_n(\Rm_{g_\infty(t)})(p_L')\geq  \frac12 L.
\end{split}
\end{equation}
Since $L>0$ is arbitrary, result follows by \eqref{ddist}.
\end{proof}

\medskip

Next we construct expander with $\Rm>0$ which its curvature does not decay to $0$ at infinity. 
\begin{proof}[Proof of Theorem \ref{thm:nodecay}]

Fix $k>1$, $\delta_i\to 0$ and let $h_i=\hat h_{\delta_i}$ be the sequence obtained from Proposition~\ref{prop:sphere-4}.  For each $i\in \mathbb{N}$, we let $(M_i,g_i, f_i,p_i)$ be the expanding gradient Ricci soliton obtained by Theorem~\ref{thm:der} with $\Rm(g_i)>0$ which is smoothly asymptotic to $(C(\mathbb{S}^{n-1}), dr^2+r^2h_i)$. Moreover, this sequence converges in Cheeger Gromov sense to a complete expanding gradient Ricci soliton $(M_\infty,g_\infty, f_\infty,p_\infty)$ with $\Rm\geq 0$ and asymptotic to $(C(\mathbb{S}^{n-1}), dr^2+r^2 h_0,o)$ in the distance sense, where $h_0=ds^2+\frac1{k^{3}}\sin^2(\sqrt{k} s)\,g_{\mathbb{S}^{n-2}}$ in the warped product coordinate $(0,\pi/\sqrt{k})\times \mathbb{S}^{n-2}$ of $\mathbb{S}^{n-1}$. We denote the two (symmetric) isolated singularities of $h_0$, i.e. $s=0,\pi/\sqrt{k}$, by $q_1,q_2$ respectively. It is clear that  $C(\mathbb{S}^{n-1})\setminus K_0$ is a smooth Riemannian manifold where 
\[
K_0:=\{o, (r,q_i), r>0, i=1,2\}.
\]

Using \cite[Theorem 1.6]{DSS2022b} as in the proof of Theorem~\ref{thm:ex1}, the induced Ricci flow $g_\infty(t)$ of the expander $g_\infty$ converges in $C^\infty_{loc}$ sense to  the cone metric $dr^2+r^2h_0$ on $C(\mathbb{S}^{n-1})\setminus K_0$ as $t\to 0^+$.  Strong maximum principle implies that $\Rm(g_\infty(t))>0$ (\cite[Lemma 2.2]{GS2018}). Since the cone is not Reifenburg on $\{(r,q_i): r>0, i=1,2\}$ w.r.t $dr^2+r^2h_0$. Result then follows from Lemma \ref{curdecay}.
\end{proof}

\medskip

Finally, we give non-flat expander whose curvature decays faster than any polynomial speed along some direction.
\begin{proof}[Proof of Theorem \ref{thm:fast}]
Fix $k>1$ and a sufficiently small $\delta,\e>0$. We let $h_{\delta,\e}$ be the metric on $\mathbb{S}^{n-1}$ obtained in sub-section~\ref{subsec:Reifenburg}. By Proposition~\ref{prop:sphere-1}, $h_{\delta,\e}=ds^2+a_{\delta,\e}^2(s)\,g_{\mathbb{S}^{n-2}}$ is spherical on $(0,\delta]\times \mathbb{S}^{n-2}$ in the warped product coordinate of $\mathbb{S}^{n-1}$. Let $C(\mathbb{S}^{n-1})$ be a metric cone with metric $g_c=dr^2+r^2 h_{\delta,\e}$ and tip $o$. Then the curvature of $g_c$ vanishes on the nonempty open set in the cone 
    \[
    U:=\{(r, q)\in C(\mathbb{S}^{n-1})\setminus \{o\}, q\in \mathbb{S}^{n-1}: s(q)\in (0,\delta)\}.
    \]
    
    Thanks to Theorem \ref{thm:der} and \cite[Lemma 2.2]{GS2018}, there exists complete expanding gradient Ricci soliton $(M,g,f,p)$ with positive curvature operator smoothly asymptotic to the cone $C(\mathbb{S}^{n-1})$. Furthermore by \cite[Theorem 4.3.1]{Siepmann2013}, the induced Ricci flow of the soliton, $g(t)$ converges to $g_c$ in $C^{\infty}_{loc}$ sense on $C(\mathbb{S}^{n-1})\setminus \{o\}$ as $t\to 0^+$, see also \cite[Theorem 1.6]{DSS2022b}. We will verify that $(M,g,f,p)$ doesn't have at most polynomial curvature decay. Suppose on the contrary. Then there exists $\ell\in  \mathbb{N}$ such that 
  \begin{equation}\label{cons}
          \liminf_{x\to\infty}  d^{\ell}_g(p,x)\cdot \mathrm{scal}_g(x)>0,
  \end{equation} where $p$ is the unique critical point of the potential $f$. Denote $d_{g_{c}}$ by $d_0$. Fix a point $x_0\in U$, there is an $r_0>0$ such that $B_{d_0}(x_0,2r_0)\subset\subset U$. Without loss of generality, we may further assume that $x_0\neq p$. As the evolution equation of the scalar curvature along a Ricci flow satisfies
  \[
  (\partial_t-\Delta_{g(t)})\mathrm{scal}_{g(t)}\leq  C_n\mathrm{scal}_{g(t)}^2
  \]
 and that $g(1)=g$ has bounded curvature, we can then apply Proposition~\ref{prop:localMP} with scaling to $\varphi(x,t):=\mathrm{scal}_{g(t)}(x)$ to see that for any $m\in \mathbb{N}$, there exists $\hat T_m>0$ such that for any $t\in (0,\hat T_m)$,
$$t\mathrm{scal}(x_0,t)=t\varphi(x_0,t)\leq  r_0^{-2m-2}t^{m+1} .$$ 

Using \eqref{ddist} and distance distortion, we have
\begin{equation*}
\begin{split}
&\quad \lim_{t\to 0} d_g(\phi_t(x_0),p)^\ell \cdot \mathrm{scal}_g(\phi_t(x_0))\\
&\leq \lim_{t\to 0}\left(\frac{(d_g(x_0,p)+C)}{\sqrt{t}}\right)^\ell \cdot t\varphi(x_0,t)\\
&\leq \lim_{t\to 0}\left(\frac{(d_g(x_0,p)+C)}{\sqrt{t}}\right)^\ell \cdot r_0^{-2m-2} t^{m+1}
\end{split}
\end{equation*}

This contradict with \eqref{cons} if we choose $m>\ell/2$,
 and completes the proof of Theorem \ref{thm:fast}.
\end{proof}

\end{document}